\documentclass{article}
\usepackage[utf8]{inputenc}

\usepackage[english]{babel}

\usepackage[utf8]{inputenc}
\usepackage[T1]{fontenc}

\usepackage{amssymb}
\usepackage{stmaryrd}
\usepackage{amsmath,amsfonts}
\usepackage[tikz]{bclogo}
\usepackage[top=3cm]{geometry}

\usepackage{lmodern}
\usepackage{textcomp}
\usepackage{mathtools, bm}
\usepackage{amssymb, bm}
\usepackage[unq]{unique}
\usepackage[shortlabels]{enumitem}
\usepackage{comment}
\usepackage[breaklinks]{hyperref}

\usepackage[numbers]{natbib}  

\usepackage[breaklinks]{hyperref}
\usepackage[numbers]{natbib}

\usepackage{amsthm}

\usepackage{cleveref}
\usepackage[utf8]{inputenc}
\usepackage[T1]{fontenc}
\usepackage{enumitem}
\usepackage{tipa}
\usepackage{float}
\usepackage{dirtytalk}
\usepackage{pbox}
\usepackage{xcolor}
\usepackage{setspace}
\usepackage[normalem]{ulem}
\usepackage{geometry}
\usepackage{amsmath}
\usepackage{amsthm}
\usepackage{tikz}
\usepackage{amssymb}
\usepackage{multicol}
\usepackage{microtype}
\usepackage{datetime}
\usepackage{endnotes}
\usepackage{babel}
\usepackage[mathscr]{euscript}
\usepackage{bigfoot}
\usetikzlibrary {arrows.meta}
\usepackage{bbm}

\newtheorem{theorem}{Theorem}[section]
\newtheorem{lemma}[theorem]{Lemma}

\newtheorem{conjecture}[theorem]{Conjecture}

\newtheorem{Cor}[theorem]{Corollary}

\newtheorem{Claim}{Claim}

\theoremstyle{definition}

\newcommand{\dyck}[1]{\mathcal{D}_{#1}}
\newcommand{\catalan}[1]{C_{#1}}
\newcommand{\groundset}[1]{\mathbb{Z}_{#1}}
\newcommand{\groundsubs}[2]{\groundset{#1}^{(#2)}}

\newcommand{\stat}[3]{\iota_{#1}(#2,#3)}

\newcommand{\paths}[2]{\mathrm{P}_{#1}^{#2}}
\newcommand{\fly}[0]{\mathcal{F}}

\title{Intervals in Dyck paths and the wreath conjecture}
\author{Jan Petr\thanks{jp895@cam.ac.uk, Department of Pure Mathematics and Mathematical Statistics (DPMMS), University of Cambridge, Wilberforce Road, Cambridge, CB3 0WA, United Kingdom} \and Pavel Turek\thanks{pkah149@live.rhul.ac.uk, Department of Mathematics, Royal Holloway, University of London, Egham, Surrey TW20 0EX, United Kingdom}}
\date{}

\begin{document} 

\maketitle

\begin{abstract}
    Let $\stat{k}{m}{l}$ denote the total number of intervals of length $m$ across all Dyck paths of semilength $k$ such that each interval contains precisely $l$ falls.
    We give the formula for $\stat{k}{m}{l}$ and show that $\stat{k}{k}{l}=\binom{k}{l}^2$.
    Motivated by this, we propose two stronger variants of the wreath conjecture due to Baranyai for $n=2k+1$.
\end{abstract}

\section{Dyck paths and the main result}\label{sec:Dyck}

A \emph{Dyck path of semilength $k$} (\emph{Dyck $k$-path} for brevity) is a lattice path in $\mathbb{Z}^2$ that never goes below the $x$-axis, starts at $(0,0)$, ends at $(2k,0)$, and with each step of the form either $(1,1)$ -- a \emph{rise} --, or $(1,-1)$ -- a \emph{fall}. We will denote the set of all Dyck $k$-paths by $\dyck{k}$.
It is well known that $|\dyck{k}|=\catalan{k}$, where $\catalan{k}=\frac{1}{k+1}\binom{2k}{k}$ is the $k$-th \emph{Catalan number}.
The Catalan numbers appear in a great many combinatorial settings; the reader is referred to \cite{Stanley_2015} for an extensive compilation.
Dyck paths have also been widely studied; see the works by Deutsch \cite{DEUTSCH1999167} and by Blanco and Petersen \cite{blanco2014counting} for a collection of statistics and other results about them.

For $D \in \dyck{k}$ and non-negative integers $l \leq m \leq 2k$, we define $\stat{D}{m}{l}$ as the number of intervals of length $m$ (that is, sequences of $m$ consecutive steps) in $D$ such that the interval contains precisely $l$ falls. For a non-negative integer $k$ we define $\stat{k}{m}{l}$ as $\sum_{D \in \dyck{k}} \stat{D}{m}{l}$. Note that by considering reflections in the line $x=k$, we have $\stat{k}{m}{l} = \stat{k}{m}{m-l}$.

Our main result is a formula for $\stat{k}{m}{l}$. Note that in view of the paragraph above, we can restrict our attention to the case $2l\leq m$.  We set $\binom{x}{y}=0$ for integers $x \geq 0$ and $y<0$.

\begin{theorem}\label{thm:main}
Let $l,m,k$ be three non-negative integers such that $2l \leq m \leq 2k$. Then
$$
\stat{k}{m}{l}=\sum_{d=0}^{k+l-m}\bigg[\binom{m}{l}-\binom{m}{l-d-1}\bigg]\bigg[\binom{2k-m+1}{k-m+l-d}-\binom{2k-m+1}{k-m+l-d-1}\bigg].
$$
\end{theorem}

A simple manipulation with the sum in \Cref{thm:main} (see the beginning of the proof of \Cref{cor:main} for a demonstration in the case $m=k$) provides an alternative formula for $\stat{k}{m}{l}$ given by
$$\stat{k}{m}{l}=\binom{m}{l}\binom{2k-m+1}{k-m+l}+  \sum_{d'=0}^{l-1}\binom{m}{d'}\bigg[\binom{2k-m+1}{k+1-d'}-\binom{2k-m+1}{k-d'}\bigg].$$

The formula becomes particularly elegant in the special case $m=k$.

\begin{Cor} \label{cor:main}
Let $l \leq k$ be two non-negative integers. Then $\stat{k}{k}{l}=\binom{k}{l}^2$.
\end{Cor}

We remark that using an analogous technique as in the proof of \Cref{cor:main} in \Cref{sec:Proofs}, one can also obtain $\stat{k}{k+1}{l}=\binom{k}{l-1}\binom{k}{l}$ for $2l\leq k+1$. With a bit more effort, similar lines of reasoning eventually yield $\stat{k}{k-1}{l}=\binom{k+2}{l+1}\binom{k-1}{l} - \binom{k}{l+1}\binom{k-1}{l-1} - \binom{k}{l-1}\binom{k-1}{l}$ for $2l\leq k-1$.

\section{The wreath conjecture}\label{sec:Wreath}

Kirkman's Schoolgirl problem from 1847 \cite{kirkman1847problem} gave rise to the question whether the $k$-uniform complete hypergraph on $n$ vertices can be partitioned into perfect matchings (that is, sets of hyperedges such that each vertex lies in exactly one of the hyperedges) whenever $k$ divides $n$.
The positive answer was confirmed by Baranyai in 1974 \cite{baranyai1974factrization}.
At the end of his paper, Baranyai posed a conjecture concerning a generalisation of his result.

This conjecture was originally stated in terms of `staircase matrices'.
Later, Katona \cite{katona1991renyi} rephrased the conjecture in terms of `wreaths'; it is this notion that we adapt here.

Let $k \leq n$ be two positive integers, and let $g=\gcd(n,k)$. We write $\groundset{n}$ for the set of integers modulo $n$ and $\groundsubs{n}{k}$ for the set of subsets of $\groundset{n}$ of size $k$.
Given a permutation $\pi$ of $\groundset{n}$, we define the set $\fly_{\pi} \subset \groundsubs{n}{k}$ as
$$\{\{\pi((i-1)k+1), \pi((i-1)k+2), \ldots, \pi(ik)\}\,|\, i\in \groundset{n}\}.$$ It is easy to see that such a set has size $\frac{n}{g}$.
We call $\fly_{\pi}$ the \emph{wreath generated by $\pi$} and say that $\fly \subset \groundsubs{n}{k}$ is a \emph{wreath} if it is a wreath generated by a permutation.

The conjecture due to Baranyai \cite{baranyai1974factrization} and Katona \cite{katona1991renyi} (who nicknamed it \emph{the wreath conjecture}) is as follows.

\begin{conjecture}[The wreath conjecture]\label{conj:original}
    For any positive integers $k \leq n$ there is a decomposition of $\groundsubs{n}{k}$ into disjoint wreaths.    
\end{conjecture}

For example, $\groundsubs{5}{2}$ can be decomposed into $\fly_{\pi_1}=\{\{0,1,2\}, \{1,2,3\}, \{2,3,4\}, \{3,4,0\}, \{4,0,1\}\}$ and $\fly_{\pi_2}=\{\{0,2,4\}, \{2,4,1\}, \{4,1,3\}, \{1,3,0\}, \{3,0,2\}\}$, where $\pi_1=\text{id}$ and $\pi_2=
\left(\begin{smallmatrix} 
0 & 1 & 2 & 3 & 4\\
0 & 2 & 4 & 1 & 3
\end{smallmatrix}\right)$.

We remark that for $n$ and $k$ coprime, this conjecture coincides with a later one due to Bailey and Stevens \cite{BAILEY20103088} concerning decompositions of complete $k$-uniform hypergraphs into tight Hamiltonian cycles. For more discussion of the case when $n$ and $k$ are not relatively prime, the reader is referred to a parallel article by the authors \cite{turek2024wreath}.

Let us consider the wreath conjecture for $n=2k+1$. In this case we have $g=1$, and hence, after rearranging, we can write $\fly_{\pi}=\{\{\pi(i+1), \pi(i+2), \ldots, \pi(i+k) \}\, | \, i\in \groundset{n} \}$. Motivated by the fact that for $n=2k+1$ the number of wreaths necessary to decompose $\groundsubs{n}{k}$ coincides with the Catalan number $C_k$, we propose the following strengthening of \Cref{conj:original}. 

\begin{conjecture}\label{conj:weaker}
Let $k$ be a positive integer.
There exists a set $\Pi=\{\pi_1, \pi_2, \ldots, \pi_{\catalan{k}}\}$ of $C_k$ permutations of $\groundset{2k+1}$ with each permutation fixing $0$ and a bijection $\varphi: \Pi \rightarrow \dyck{k}$ such that
\begin{itemize}
    \item $\groundsubs{2k+1}{k}=\bigcup_{i=1}^{\catalan{k}}\fly_{\pi_{i}}$, and
    \item for any $i$ and $j$, the $j$-th step of $\varphi(\pi_i)$ is a rise if and only if $\pi_i(j) \in \{1,2,\ldots,k\}$.
\end{itemize}
\end{conjecture}

We verified the conjectures using a computer for $k\leq 4$. The bijections for $k\leq 3$ can be seen in \Cref{wreaths}. This conjecture is also motivated by \Cref{cor:main} in view of the following result.

\begin{lemma}\label{le:necessary}
    The equality $\stat{k}{k}{l}= \binom{k}{l}^2$ from \Cref{cor:main} is a necessary condition for \Cref{conj:weaker}.
\end{lemma}

\begin{proof}
If \Cref{conj:weaker} holds, then to each interval $I$ of length $k$ of a Dyck $k$-path $D$ we can assign a set $\{ \pi(s), \pi(s+1), \dots, \pi(s+k-1)\}$, where $I$ starts at the $s$-th step of $D$ and $\pi = \varphi^{-1}(D)$.

By the first condition of \Cref{conj:weaker}, this yields a bijection between intervals of length $k$ of Dyck $k$-paths and sets in $\groundsubs{2k+1}{k}$ not containing $0$. 
The second condition then implies that intervals of length $k$ with $l$ falls are in bijection with sets in $\groundsubs{2k+1}{k}$ without $0$ and with $l$ elements from $\{ k+1,k+2,\ldots,2k\}$. From the definition, there are $\stat{k}{k}{l}$ intervals of length $k$ with $l$ falls and there are $\binom{k}{l}^2$ sets in $\groundsubs{2k+1}{k}$ without $0$ and with $l$ elements from $\{ k+1,k+2,\ldots,2k\}$. The result follows.
\end{proof}

Our search through $k \leq 4$ suggests that an even stronger statement may be true. To state it, for a Dyck path $D \in \dyck{k}$, we will denote by $D^{(R)}$ the Dyck path obtained by reflecting $D$ in the line $x=k$.

\begin{conjecture}\label{conj:stronger}
Let $k$ be a positive integer. There exists a set $\Pi=\{\pi_1, \pi_2, \ldots, \pi_{\catalan{k}}\}$ of $C_k$ permutations of $\groundset{2k+1}$ with each permutation fixing $0$ and a bijection $\varphi: \Pi \rightarrow \dyck{k}$ such that
\begin{itemize}
    \item $\groundsubs{2k+1}{k}=\bigcup_{i=1}^{\catalan{k}}\fly_{\pi_{i}}$, and
    \item for any $i$ and $j$, the $j$-th step of $\varphi(\pi_i)$ is a rise if and only if $\pi_i(j) \in \{1,2,\ldots,k\}$, and
    \item for any Dyck $k$-path $D$ and any $j \in \groundset{2k+1}$ we have $\varphi^{-1}(D)(j)+\varphi^{-1}(D^{(R)})(-j)=0$.

\end{itemize}
\end{conjecture}

The bijections from \Cref{wreaths} all satisfy the stronger \Cref{conj:stronger}. 

\begin{figure}[h]\centering
    			\includegraphics[height=7 cm]{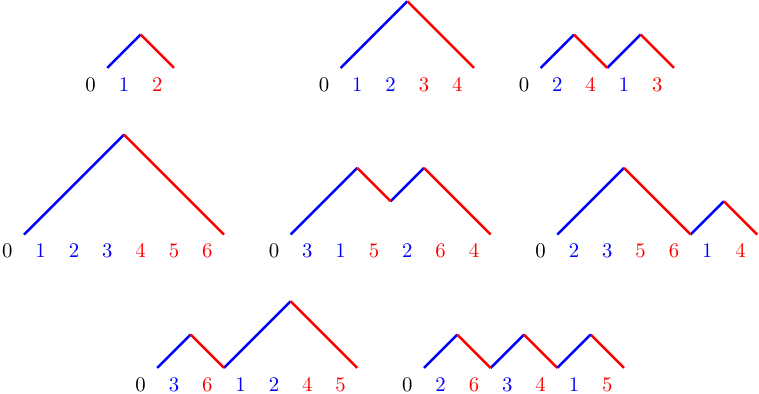}
    			\caption{Examples of permutations which confirm Conjectures \ref{conj:weaker} and \ref{conj:stronger} for $k \leq 3$. Below each Dyck $k$-path the corresponding permutation $\pi$ is written in the form $\pi(0), \pi(1), \dots, \pi(2k)$. The second conditions of the conjectures require numbers $1,2,\dots, k$ to lie below rises and numbers $k+1, k+2, \dots, 2k$ below falls.} 
\label{wreaths}
	\end{figure}

\section{Proofs of \Cref{thm:main} and \Cref{cor:main}}\label{sec:Proofs}

To simplify the calculations in the proofs later, we introduce some additional notation. We call a lattice path a \emph{NE upper path} if it never visits a point below the $y=x$ diagonal and each of its steps is either $(0,1)$ -- a \emph{North} step --, or $(1,0)$ -- an \emph{East} step. Instead of working with Dyck $k$-paths as defined in \Cref{sec:Dyck}, we will work with NE upper paths from $(0,0)$ to $(k,k)$.
To observe that there is a bijection between these two families of paths, consider a rotation of the plane by $\frac{\pi}{4}$ and an appropriate scaling.
Under this bijection, rises translate to North steps and falls translate to East steps. 

The proof uses the following well-known generalisation of Catalan numbers counting the number of NE upper paths between $(x_1,y_1),(x_2,y_2) \in \mathbb{Z}^2$. This number, which we denote by $\paths{(x_1,y_1)}{(x_2,y_2)}$, is nonzero if and only if $x_1 \leq x_2$, $y_1 \leq y_2$ and $x_i \leq y_i$ for $i \in \{1,2\}$.
Note that for $j \in \mathbb{Z}$ we have $\paths{(x_1,y_1)}{(x_2,y_2)}=\paths{(x_1+j,y_1+j)}{(x_2+j,y_2+j)}$.

\begin{lemma}\label{lem:main}
    Let $(x_1,y_1),(x_2,y_2) \in \mathbb{Z}^2$ be such that $x_1 \leq x_2$, $y_1 \leq y_2$ and $x_i \leq y_i$ for $i \in \{1,2\}$.
    Set $\delta=y_1-x_1$, $\alpha=y_2-y_1$ and $\beta=x_2-x_1$.
    Then $$\paths{(x_1,y_1)}{(x_2,y_2)}=\binom{\alpha+\beta}{\beta}-\binom{\alpha+\beta}{\beta-\delta-1}.$$
\end{lemma}

We give a proof of this lemma to keep the paper self-contained. The proof is analogous to Andr{\'e}'s reflection principle \cite{AndreReflection87}.

\begin{proof}
    There are $\binom{\alpha+\beta}{\beta}$ paths from $(x_1,y_1)$ to $(x_2,y_2)=(x_1+\beta,y_1+\alpha)$ consisting of North and East steps.
    We claim that the number of paths from $(x_1,y_1)$ to $(x_2,y_2)$ consisting of North and East steps which visit a point below the diagonal $y=x$ is $\binom{\alpha+\beta}{\beta-\delta-1}$, which gives the result.

    To show this claim, we find a bijection between the paths from $(x_1,y_1)$ to $(x_2,y_2)$ consisting of North and East steps which visit a point below the diagonal $y=x$ and the paths from $(x_1,y_1)$ to $(x_1+\alpha+\delta+1,y_1+\beta-\delta-1)$ consisting of North and East steps.

    Given a path $W$ of the first type, consider the step after which the path visits a point below the diagonal $y=x$ for the first time. Let this be the $s$-th step. From the $(s+1)$-st step onward, we exchange North and East steps. The resulting path consists of $\beta-\delta-1$ North and $\alpha+\delta+1$ East steps, i.e., it is a path of the second type.

    Now consider a path $W'$ of the second type. Recall that $x_1 \leq y_1$, in other words, $(x_1,y_1)$ is not below the diagonal $y=x$. We also have $y_1+\beta-\delta-1=x_2-1<y_2+1=y_1+\alpha+1=x_1+\alpha+\delta+1$, therefore $(x_1+\alpha+\delta+1,y_1+\beta-\delta-1)$ is below the diagonal $y=x$. Therefore, there exists a step of $W'$ after which $W'$ visits a point below the diagonal $y=x$ for the first time. Analogously to before, exchange all North and East steps after this step. The resulting path consists of $\alpha$ North and $\beta$ East steps and visits a point below the diagonal $y=x$, therefore is of the first type.

    The maps from the previous two paragraphs are inverses to each other, and so describe the desired bijection. See \Cref{fig:lemma} for an illustration.

    \begin{figure}[hbtp]\centering
    			\includegraphics[height=4 cm]{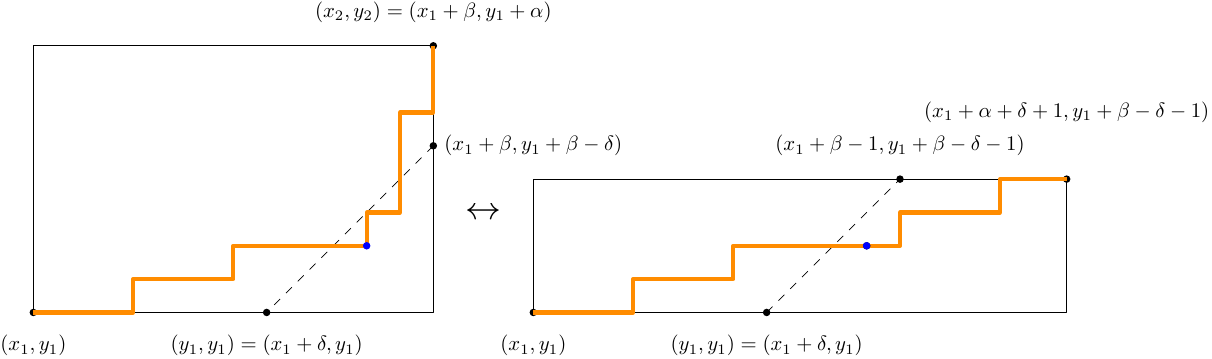}
    			\caption{The bijection from \Cref{lem:main} given by `flipping' the section of the path after the highlighted point.}
       \label{fig:lemma}
	\end{figure}

     Hence we have $\paths{(x_1,y_1)}{(x_2,y_2)}=\binom{\alpha+\beta}{\beta}-\binom{\alpha+\beta}{\beta-\delta-1}$, as claimed.
\end{proof}

Equipped with this lemma, we move on to the proof of the main theorem.

\begin{proof}[Proof of \Cref{thm:main}]
We consider all possible `starting points' $(i,i+d)$ of intervals of length $m$ which contain exactly $l$ East steps. To obtain $\stat{k}{m}{l}$, we sum over all such starting points the number of NE upper paths which visit the starting point and have exactly $l$ East steps among the next $m$ steps following this visit. Note that after these $m$ steps, any such NE upper path visits $(i+l,i+d+m-l)$. See \Cref{fig:3parts} for an illustration.

\pagebreak

\begin{figure}[hbtp]\centering
    			\includegraphics[height=9 cm]{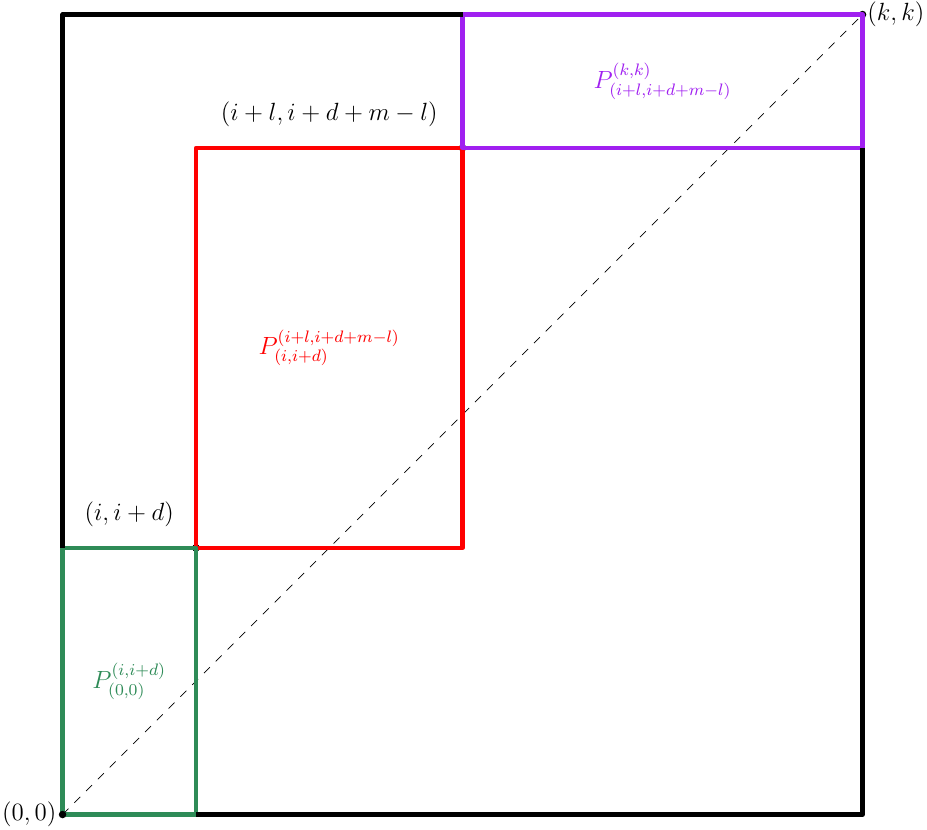}
    			\caption{An illustration of NE upper paths whose interval of length $m$ starting at $(i,i+d)$ contains exactly $l$ East steps.}
       \label{fig:3parts}
	\end{figure}

As both the starting point and the corresponding endpoint lie in the square $[0,k]^2$, we have $i,d\geq 0$ as well as $i+l\leq k$ and $i+d+m-l\leq k$. Therefore $i$ and $d$ satisfy $0 \leq d \leq k+l-m$ and $0\leq i \leq k+l-m-d$.

In the identities below, first we use $\paths{(x_1,y_1)}{(x_2,y_2)}=\paths{(x_1+j,y_1+j)}{(x_2+j,y_2+j)}$ and then simplify the sums:

\begin{align*}
    \stat{k}{m}{l}
    &=\sum_{d=0}^{k+l-m}\sum_{i=0}^{k+l-m-d}\paths{(0,0)}{(i,i+d)}\paths{(i,i+d)}{(i+l,i+d+m-l)}\paths{(i+l,i+d+m-l)}{(k,k)}\\
    &=\sum_{d=0}^{k+l-m}\sum_{i=0}^{k+l-m-d}\paths{(0,0)}{(i,i+d)}\paths{(0,d)}{(l,d+m-l)}\paths{(i+l,i+d+m-l)}{(k,k)}\\
    &=\sum_{d=0}^{k+l-m}\paths{(0,d)}{(l,d+m-l)}
    \sum_{i=0}^{k+l-m-d}\paths{(0,0)}{(i,i+d)}\paths{(i+l,i+d+m-l)}{(k,k)}.
\end{align*}

Further simplification of this sum comes from the following claim.

\begin{Claim}\label{Claim}
For any $d$ such that $0 \leq d \leq k+l-m$ we have
\begin{equation}
    \sum_{i=0}^{k+l-m-d}\paths{(0,0)}{(i,i+d)}\paths{(i+l,i+d+m-l)}{(k,k)}=\binom{2k-m+1}{k-m+l-d}-\binom{2k-m+1}{k-m+l-d-1}.\label{eq:claim}    
\end{equation}
\end{Claim}

\begin{proof}[Proof of \Cref{Claim}]

    Consider the pairs $(W,i)$, where $i \in \{ 0,1, \dots, k+l-m-d \}$ and $W$ is a NE upper path from $(l-k,d+m-l-k)$ to $(0,d)$ which goes through the point $(-i,-i)$. Using $\paths{(x_1,y_1)}{(x_2,y_2)}=\paths{(x_1+j,y_1+j)}{(x_2+j,y_2+j)}$, observe that the number of distict pairs $(W,i)$ is precisely the left-hand side of equation \eqref{eq:claim}. See \Cref{fig:2parts} for an illustration.

\begin{figure}[hbtp]\centering
    			\includegraphics[height=6 cm]{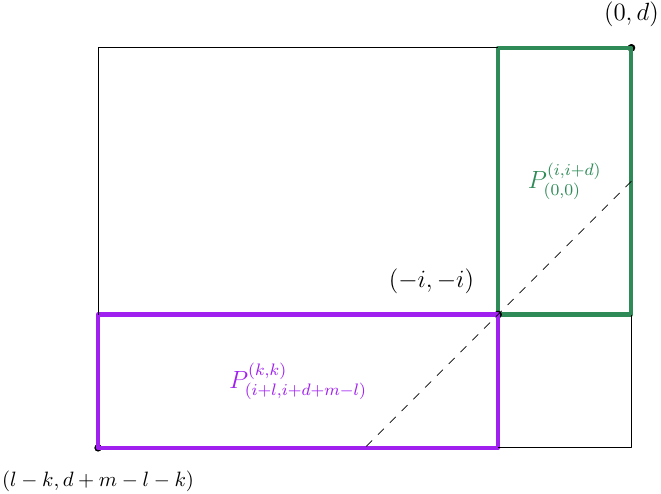}
    			\caption{An illustration of the left-hand side of equation \eqref{eq:claim}.}
       \label{fig:2parts}
	\end{figure}

 Next, we claim that the pairs $(W,i)$ are in bijection with NE upper paths from $(l-k-1,d+m-l-k)$ to $(0,d)$ which visit the diagonal $y=x$ in at least one point.
 The bijection is given as follows.
 
 Given a pair $(W,i)$, consider the sequence of $k+l-m$ North and $k-l$ East steps corresponding to $W$. After the step at which $W$ reaches $(-i,-i)$, add an additional East step. Denote by $W'$ the path corresponding to the longer sequence of North and East steps starting at $(l-k-1,d+m-l-k)$.
 Observe that $W'$ consists of $k+l-m$ North steps and $k-l+1$ East steps, never visits a point below the diagonal $y=x$ but visits a point on this diagonal.
 Moreover, the first visited point on this diagonal is $(-i,-i)$.

 On the other hand, let $W'$ be a path from $(l-k-1,d+m-l-k)$ to $(0,d)$ consisting of North and East steps that never goes below the diagonal $y=x$ but which visits this diagonal in at least one point.
 Then there is a first point at which $W'$ visits this diagonal; let this point be $(-i,-i)$.
 Consider the sequence of $k+l-m$ North and $k-l+1$ East steps corresponding to $W'$.
 Remove the East step that leads to $(-i,-i)$ to obtain a shorter sequence of steps, and denote by $W$ the corresponding path starting at $(l-k,d+m-l-k)$. Then $(W,i)$ is a pair where $i \in \{ 0,1, \dots, k+l-m-d \}$ and $W$ is a NE upper path from $(l-k,d+m-l-k)$ to $(0,d)$ which goes through the point $(-i,-i)$.

The maps from the previous two paragraphs are inverses to each other, and hence describe the desired bijection, see \Cref{fig:bijection} for an illustration.

\begin{figure}[hbtp]\centering
    			\includegraphics[height=5.5 cm]{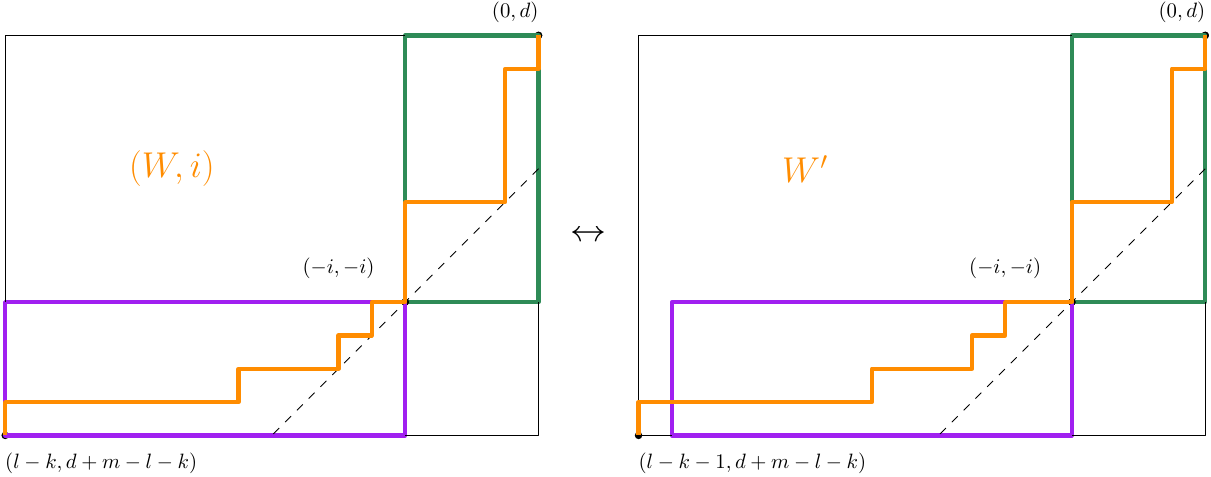}
    			\caption{The bijection from \Cref{Claim}. The additional step in $W'$ is the East step leading to $(-i,i)$.} 
\label{fig:bijection}
	\end{figure}

We thus obtain that the left-hand side of equation \eqref{eq:claim} counts the number of NE upper paths from $(l-k-1,d+m-l-k)$ to $(0,d)$ which visit the diagonal $y=x$ in at least one point. Note that such paths are exactly the NE upper paths from $(l-k-1,d+m-l-k)$ to $(0,d)$ which visit a point below the diagonal $y=x+1$. The number of such paths is $\paths{(l-k-1,d+m-l-k)}{(0,d)}-\paths{(l-k,d+m-l-k)}{(1,d)}$.

Applying \Cref{lem:main} and observing that $\paths{(l-k,m-l-k)}{(1,0)}=0=\binom{2k-m+1}{k-l+1}-\binom{2k-m+1}{(k-l+1)-(m-2l)-1}$ we get

\begin{align*}
    \sum_{i=0}^{k+l-m-d}\paths{(0,0)}{(i,i+d)}\paths{(i+l,i+d+m-l)}{(k,k)}&=
    \paths{(l-k-1,d+m-l-k)}{(0,d)}-\paths{(l-k,d+m-l-k)}{(1,d)}\\
    &=\bigg[\binom{2k-m+1}{k-l+1}-\binom{2k-m+1}{(k-l+1)-(d+m-2l+1)-1}\bigg]\\
    &\qquad-\bigg[\binom{2k-m+1}{k-l+1}-\binom{2k-m+1}{(k-l+1)-(d+m-2l)-1}\bigg]\\
    &=\binom{2k-m+1}{k-m+l-d}-\binom{2k-m+1}{k-m+l-d-1}.
\end{align*}

This concludes the proof of \Cref{Claim}.
\end{proof}

By applying \Cref{lem:main} and \Cref{Claim}, we obtain
\begin{align*}
    \stat{k}{m}{l}&=\sum_{d=0}^{k+l-m}\paths{(0,d)}{(l,d+m-l)}
    \sum_{i=0}^{k+l-m-d}\paths{(0,0)}{(i,i+d)}\paths{(i+l,i+d+m-l)}{(k,k)}\\
    &=\sum_{d=0}^{k+l-m}\bigg[\binom{m}{l}-\binom{m}{l-d-1}\bigg]\bigg[\binom{2k-m+1}{k-m+l-d}-\binom{2k-m+1}{k-m+l-d-1}\bigg],
\end{align*}
completing our proof of \Cref{thm:main}.
\end{proof}

The formula from \Cref{thm:main} simplifies considerably when $m=k$.

\begin{proof}[Proof of \Cref{cor:main}]
   Plugging $m=k$ into \Cref{thm:main} we obtain
   \begin{align*}
   \stat{k}{k}{l}&=\sum_{d=0}^{l}\bigg[\binom{k}{l}-\binom{k}{l-d-1}\bigg]\bigg[\binom{k+1}{l-d}-\binom{k+1}{l-d-1}\bigg]\\
   &=\binom{k}{l}\sum_{d=0}^{l}\bigg[\binom{k+1}{l-d}-\binom{k+1}{l-d-1}\bigg]-\sum_{d=0}^{l-1}\binom{k}{l-d-1}\bigg[\binom{k+1}{l-d}-\binom{k+1}{l-d-1}\bigg].    
   \end{align*}

Noting that the second factor in the first term is a telescoping sum, we find that
\begin{equation}\label{eq2}
   \binom{k}{l}\sum_{d=0}^{l}\bigg[\binom{k+1}{l-d}-\binom{k+1}{l-d-1}\bigg]=\binom{k}{l}\binom{k+1}{l}. 
\end{equation}

To simplify the second term, first rewrite it as
\begin{align*}
&\sum_{d=0}^{l-1}\binom{k}{l-d-1}\bigg[\binom{k+1}{l-d}-\binom{k+1}{l-d-1}\bigg]=\\
&\qquad\sum_{d'=0}^{l-1}\binom{k}{d'}\binom{k+1}{k-d'}-\sum_{d'=0}^{l-1}\binom{k+1}{d'}\binom{k}{k-d'}.
\end{align*}

Consider now the tasks of choosing $k$ objects out of $2k+1$ ordered objects. The first sum counts the number of ways to do so in a way that at most $l-1$ of the first $k$ objects are picked. The second sum counts the number of ways to do so in a way that at most $l-1$ of the first $k+1$ objects are picked. Any choice of the second type is a choice of the first type. The only choices of the first type that are not of the second type are those which choose exactly $l-1$ elements from the first $k$ and also choose the $(k+1)$-st element (which means $k-l$ elements are chosen from the last $k$ elements).

    We thus obtain

\begin{equation}\label{eq3}
   \sum_{d=0}^{l-1}\binom{k}{l-d-1}\bigg[\binom{k+1}{l-d}-\binom{k+1}{l-d-1}\bigg]=\binom{k}{l-1}\binom{k}{k-l}=\binom{k}{l}\binom{k}{l-1}. 
\end{equation}

   Equations \eqref{eq2} and \eqref{eq3} imply
   $$\stat{k}{k}{l}=\binom{k}{l}\binom{k+1}{l}-\binom{k}{l}\binom{k}{l-1}=\binom{k}{l}^2,$$
   as claimed.
\end{proof}

\section{Concluding remarks}
The main open questions of interest are the conjectures from \Cref{sec:Wreath}. As we have seen, \Cref{cor:main} provides support for the newly introduced \Cref{conj:weaker} and \Cref{conj:stronger}.
The authors are also curious if the right-hand side of \Cref{thm:main} could be simplified further in other cases than those with $m = k, k-1, k+1$ already considered.

\section*{Acknowledgements}
The authors thank Miroslav Olšák for observing \Cref{le:necessary} which motivated the main results of this paper. They also thank Béla Bollobás and Mark Wildon for helpful comments concerning the project and the manuscript. 

The first author would like to acknowledge support by the EPSRC (Engineering and Physical Sciences Research Council), reference EP/V52024X/1, and by the Department of Pure Mathematics and Mathematical Statistics of the University of Cambridge. The second author would like to acknowledge support from Royal Holloway, University of London.

\bibliographystyle{abbrvnat}  
\bibliography{bibliography}

\end{document}